 \documentclass[preprint,12pt]{elsarticle}

\usepackage{amssymb}
\usepackage{amsfonts}
\usepackage{amsmath}
\usepackage{graphicx}%
\newtheorem{theorem}{Theorem}

\newtheorem{corollary}{Corollary}

\newtheorem{definition}{Definition}
\newtheorem{example}{Example}

\newtheorem{lemma}{Lemma}

\newtheorem{proposition}{Proposition}
\newtheorem{remark}{Remark}

\newproof{proof}{Proof}
\numberwithin{equation}{section}

\journal{arxiv.org}

\begin{document}

\begin{frontmatter}



\title{Halanay type inequalities on time scales with applications}


\author[MA]{Murat Ad\i var \corref{cor1}}
\cortext[cor1]{Corresponding Author}
\address[MA]{Izmir University of Economics\\
Department of Mathematics, 35330, Izmir, Turkey}
\ead{murat.adivar@ieu.edu.tr}
\author[EB]{Elvan Ak\i n Bohner}
\address[EB]{Missouri University Science \& Technology \\ Department of Mathematics \&
Statistics, Rolla, MO, 65409-0020, USA}
\ead{akine@mst.edu}

\begin{abstract}
This paper aims to introduce Halanay type inequalities on time scales. By
means of these inequalities we derive new global stability conditions for
nonlinear dynamic equations on time scales. Giving several examples we show
that beside generalization and extension to $q$-difference case, our results
also provide improvements for the existing theory regarding differential and
difference inequalites, which are the most important particular cases of
dynamic inequalities on time scales.

\end{abstract}


\end{frontmatter}


\section{Introduction and preliminaries}

Stability analysis of dynamical systems using differential and difference
inequalities attracted a prominent attention in the existing literature (see
\cite{agarwal 1}-\cite{gopalsamy} and references therein). For stability
analysis of the delay differential equation%
\[
x^{\prime}(t)=-px(t)+qx(t-\tau),\ \ \tau>0,
\]
Halanay proved the following result.

\begin{lemma}
[Halanay, 1966]\cite{halanay}\label{lem halanay} If
\[
f^{\prime}(t)\leq-\alpha f(t)+\beta\sup_{s\in\left[  t-\tau,t\right]
}f(s)\ \ \text{for\ \ }t\geq t_{0}%
\]
and $\alpha>\beta>0$, then there exist $\gamma>0$ and $K>0$ such that
\[
f(t)\leq Ke^{-\gamma(t-t_{0})}\ \ \text{for\ \ }t\geq t_{0}.
\]

\end{lemma}

In 2000, Mohamad and Gopalsamy gave the next theorem:

\begin{theorem}
\cite{gopalsamy}\label{gopalsamy} Let $x$ be a nonnegative function
satisfying
\[
x^{\prime}(t)\leq-a(t)x(t)+b(t)\left(  \sup_{s\in\left[  t-\tau(t),t\right]
}x(s)\right)  \text{, }t\geq t_{0}%
\]%
\[
x(s)=\left\vert \varphi(s)\right\vert \text{ for }s\in\left[  t_{0}-\tau
^{\ast},t_{0}\right]  ,
\]
where $\tau(t)$ denotes a nonnegative continuous and bounded function defined
for $t\in\mathbb{R}$ and $\tau^{\ast}=\sup\limits_{t\in\mathbb{R}}\tau(t)$;
$\varphi(s)$ is continuous and defined on $\left[  t_{0}-\tau^{\ast}%
,t_{0}\right]  $; $a(t)$ and $b(t)$, $t\in\mathbb{R}$, denote nonnegative,
continuous and bounded functions. Suppose%
\[
a(t)-b(t)\geq L\text{, }t\in\mathbb{R}\text{,}%
\]
where $L=\inf\limits_{t\in\mathbb{R}}\left(  a(t)-b(t)\right)  >0$. Then there
exists a positive number $\lambda$ such that%
\[
x(t)\leq\left(  \sup_{s\in\left[  t_{0}-\tau^{\ast},t_{0}\right]
}x(s)\right)  e^{-\lambda(t-t_{0})}\text{ for }t>t_{0}\text{.}%
\]

\end{theorem}

Afterwards, numerous variants of Halanay's inequality have been treated in the
literature. Stability analysis of differential equations using Halanay type
inequalities has been studied in \cite{halanay}, \cite{Ivanov}, and \cite{wang}. For stability
analysis of difference equations using Halanay inequality one may consult with
\cite{agarwal}-\cite{liz}. A comprehensive review on the recent developments
in discrete and continuous Halanay type inequalities can be found in
\cite{baker} and \cite{baker tang}. A time scale is an arbitrary nonempty
closed subset of reals. Stability analysis of dynamics equations on time
scales using Lyapunov functionals has been studied in \cite{adivar1}%
-\cite{raffoul2}. To the best of our knowledge, Halanay type inequalities on
time scales and stability analysis using them have not been investigated
elsewhere before this study. One of the aims of this paper is to fill this gap
and show how Halanay inequalities on time scales can be used for the stability
analysis of dynamic equations.

In this paper, we employ the shift operators $\delta_{\pm}$ to construct delay
dynamic inequalities on time scales. Using these dynamic inequalities we
derive Halanay type inequalities for dynamic equations on time scales. By
means of Halanay inequalities and the properties of exponential function on
time scales (see Lemma \ref{remark osc}) we propose new conditions that lead
to stability for nonlinear dynamic equations on time scales. Main contribution
of this paper can be outlined as follows:

\begin{itemize}
\item Construction of Halanay type inequalities on time scales,

\item Investigation of global stability of delay dynamic equations on time
scales using Halanay inequality,

\item Improvement of the existing results for differential and difference
equations which are the most important particular cases of our problem (we
highlight this improvement by Remarks \ref{rem imp1}, \ref{rem imp 2}, and
\ref{rem imp 3}).
\end{itemize}

In \cite{yang cao}, Halanay inequalities are used to derive sufficient
conditions for the existence of periodic solutions of delayed cellular neural
networks with impulsive effects. Motivated by the study \cite{yang cao}, we
note that the results obtained in this paper can also be employed in another
research regarding the derivation of sufficient conditions for the existence
of (uniformly asymptotically stable) periodic solutions of some nonlinear
scalar systems on time scales.

We organize the paper as follows: First and second sections are devoted to
preliminary results of theory of time scales and shift operators on time
scales, respectively. In the third section, we use the shift operators on time
scales to construct delay functions and a general form of delay dynamic
equations, and obtain some dynamic inequalities. We finalize our study by
providing sufficient conditions for stability of nonlinear dynamic equations
on time scales.

Hereafter, we give some basic results that will be used in our further analysis.

To indicate a time scale (a nonempty closed subset of reals) we use the
notation $\mathbb{T}$. We classify the points of a time scale $\mathbb{T}$ by
using the forward jump and backward jump operators defined by%
\begin{equation}
\sigma(t):=\inf\left\{  s\in\mathbb{T}:s>t\right\}  \label{sigma}%
\end{equation}
and%
\[
\rho(t):=\sup\left\{  s\in\mathbb{T}:s<t\right\}  ,
\]
respectively. A point $t$ in $\mathbb{T}$ is said to be right-scattered
(right-dense) if $\sigma(t)>t$ ($\sigma(t)=t)$. We say $t\in\mathbb{T}$ is
left-scattered (left-dense) if $\rho(t)<t$ ($\rho(t)=t$). If $\rho
(t)<t<\sigma(t)$, then $t\in\mathbb{T}$ is called isolated point. The set
$\mathbb{T}^{\kappa}$ is derived from the time scale $\mathbb{T}$ as follows:
If $\mathbb{T}$ has a left-scattered maximum $m$, then $\mathbb{T}^{\kappa
}=\mathbb{T-}\left\{  m\right\}  $. Otherwise $\mathbb{T}^{\kappa}=\mathbb{T}%
$. The delta derivative of a function $f:\mathbb{T\rightarrow R}$, defined at
a point $t\in\mathbb{T}^{\kappa}$ by%
\begin{equation}
f^{\Delta}(t):=\lim_{\substack{s\rightarrow t\\s\neq\sigma(t)}}\frac
{f(\sigma(t))-f(s)}{\sigma(t)-s}\text{,} \label{delta derivative}%
\end{equation}
was first introduced by Hilger \cite{hilger} to unify discrete and continuous
analyses. It follows from the definition of the operator $\sigma$ that%
\begin{equation}
\sigma(t)=\left\{
\begin{array}
[c]{cc}%
t & \text{if }\mathbb{T=R}\\
t+1 & \text{if }\mathbb{T=Z}\\
qt & \text{if }\mathbb{T=}\overline{q^{\mathbb{Z}}}\\
t+h & \text{if }\mathbb{T=}h\mathbb{Z}%
\end{array}
\right.  , \label{forward jump}%
\end{equation}
where $\overline{q^{\mathbb{Z}}}=\{q^{k}:k\in\mathbb{Z\ }$and\ $q>1\}\cup
\left\{  0\right\}  $ and $h\mathbb{Z=\{}hn:n\in\mathbb{Z}$ and
$h>0\mathbb{\}}$. Hence, the delta derivative $f^{\Delta}(t)$ turns into
ordinary derivative $f^{\prime}(t)$ if $\mathbb{T}=\mathbb{R}$ and it becomes
the forward $h-$difference operator $\Delta_{h}f(t):=\frac{1}{h}[f(t+h)-f(t)]$
whenever $\mathbb{T}=h\mathbb{Z}$ (i.e. $f^{\Delta}(t)=f(t+1)-f(t)$=$\Delta
f(t)$ if $h=1$). For the time scale $\mathbb{T=}\overline{q^{\mathbb{Z}}}$ we
have $f^{\Delta}(t)=D_{q}f(t)$, where
\begin{equation}
D_{q}f(t)=\frac{f(qt)-f(t)}{(q-1)t}. \label{q derivative}%
\end{equation}

It follows from (\ref{delta derivative}) and (\ref{forward jump}) that dynamic
equations on time scales turn into difference equations when the time scale is
chosen as the set of integers, and they become differential equations when the
time scale coincides with the set of reals. Moreover, $q$-difference,
$h$-difference equations, used in the discretization of differential
equations, are all particular cases of dynamic equations on time scales. Since
there are many time scales other than the sets of reals and integers, analysis
on time scales provides a more general theory which enables us to see
similarities and differences between the analyses on discrete and continuous
time domains.

Throughout the paper, we denote by $[a,b]_{\mathbb{T}}$ the closed time scale
interval $[a,b]\cap\mathbb{T}$. The other time scale intervals
$[a,b)_{\mathbb{T}}$, $(a,b]_{\mathbb{T}}$, and $(a,b)_{\mathbb{T}}$ are
defined similarly. A function $f:\mathbb{T}\rightarrow\mathbb{R}$ is called
$rd$-continuous if it is continuous at right dense points and its left sided
limits exists (finite) at left dense points. The set of $rd$-continuous
functions $f:\mathbb{T\rightarrow R}$ is denoted by $C_{rd}=C_{rd}%
(\mathbb{T)}$. It is known by \cite[Theorem 1.60]{book} that the forward jump
operator defined by (\ref{sigma}) is an $rd$-continuous. By \cite[Theorem
1.65]{book} it is concluded that every $rd$-continuous function is bounded on
a compact interval. Note that continuity implies $rd$-continuity. Every
$rd$-continuous function $f:\mathbb{T}\rightarrow\mathbb{R}$ has an
anti-derivative%
\[
F(t)=\int_{t_{0}}^{t}f(t)\Delta t.
\]
That is, $F^{\Delta}(t)=f(t)$ for all $t\in\mathbb{T}^{\kappa}$ (see
\cite[Theorem 1.74]{book2}). For an excellent review on $\Delta$-derivative
and $\Delta$-Riemann integral we refer the reader to \cite{book}.

Hereafter, we give some basic definitions and theorems that will be used in
further sections.

\begin{definition}
A function $h:\mathbb{T}\rightarrow\mathbb{R}$ is said to be \emph{regressive}
provided $1+\mu(t)h(t)\neq0$ for all $t\in\mathbb{T}^{\kappa}$, where
$\mu(t)=\sigma(t)-t$. The set of all regressive $rd$-continuous functions
$\varphi:\mathbb{T}\rightarrow\mathbb{R}$ is denoted by $\mathcal{R}$ while
the set $\mathcal{R}^{+}$ is given by $\mathcal{R}^{+}=\{h\in\mathcal{R}%
:1+\mu(t)\varphi(t)>0\mbox{
for all }t\in\mathbb{T}\}$.
\end{definition}

Let $\varphi\in\mathcal{R}$. The \emph{exponential function} on $\mathbb{T}$
is defined by
\begin{equation}
e_{\varphi}(t,s)=\exp\left(  \int_{s}^{t}\!\zeta_{\mu(r)}(\varphi(r))\Delta
r\right)  \label{exp}%
\end{equation}
where $\zeta_{\mu(s)}$ is the cylinder transformation given by
\begin{equation}
\zeta_{\mu(r)}(\varphi(r))\!:=\left\{
\begin{array}
[c]{cc}%
\frac{1}{\mu(r)}\mbox{Log}(1+\mu(r)\varphi(r)) & if\text{ }\mu(r)>0\\
\varphi(r) & if\ \ \mu(r)=0
\end{array}
\right.  \,. \label{cylinder}%
\end{equation}
It is well known that (see \cite[Theorem 14]{akin}) if $p\in\mathcal{R}^{+}$,
then $e_{p}(t,s)>0$ for all $t\in\mathbb{T}$. Also, the exponential function
$y(t)=e_{p}(t,s)$ is the solution to the initial value problem $y^{\Delta
}=p(t)y,\,y(s)=1$. Other properties of the exponential function are given in
the following results:

\begin{lemma}
\label{lemma2.3} \cite[Theorem 2.36]{book} Let $p,q\in\mathcal{R}$. Then

\begin{itemize}
\item[i.] $e_{0}(t,s)\equiv1$ and $e_{p}(t,t)\equiv1$;

\item[ii.] $e_{p}(\sigma(t),s)=(1+\mu(t)p(t))e_{p}(t,s)$;

\item[iii.] $\frac{1}{e_{p}(t,s)}=e_{\ominus p}(t,s)$ where $\ominus
p(t)=-\frac{p(t)}{1+\mu(t)p(t)}$;

\item[iv.] $e_{p}(t,s)=\frac{1}{e_{p}(s,t)}=e_{\ominus p}(s,t)$;

\item[v.] $e_{p}(t,s)e_{p}(s,r)=e_{p}(t,r)$;

\item[vi.] $\left(  \frac{1}{e_{p}(\cdot,s)}\right)  ^{\Delta}=-\frac
{p(t)}{e_{p}^{\sigma}(\cdot,s)}$.
\end{itemize}
\end{lemma}

\begin{lemma}
\label{remark osc} \cite{oscillation}For a nonnegative $\varphi$ with
$-\varphi\in\mathcal{R}^{+}$, we have the inequalities%
\[
1-\int_{s}^{t}\varphi(u)\leq e_{-\varphi}(t,s)\leq\exp\left\{  -\int_{s}%
^{t}\varphi(u)\right\}  \text{ for all }t\geq s.
\]
If $\varphi$ is $rd$-continuous and nonnegative, then%
\[
1+\int_{s}^{t}\varphi(u)\leq e_{\varphi}(t,s)\leq\exp\left\{  \int_{s}%
^{t}\varphi(u)\right\}  \text{ for all }t\geq s.
\]

\end{lemma}

\begin{remark}
\cite[Remark 2.12]{akin2}\label{Lemma ep} If $\lambda\in\mathcal{R}^{+}$ and
$\lambda(r)<0$ for all $t\in\lbrack s,t)_{\mathbb{T}}$, then%
\[
0<e_{\lambda}(t,s)\leq\exp\left(  \int_{s}^{t}\lambda(r)\Delta r\right)  <1.
\]

\end{remark}

\section{Shift Operators and Delay functions}

\subsection{Shift operators}

First, we give a generalized version of shift operators (see \cite{adrafdelay} and
\cite{bams}). A limited version of shift operators can be found in
\cite{adivar}.

\begin{definition}
[Shift operators]\cite{adrafdelay}\label{shift} Let $\mathbb{T}^{\ast}$ be a
non-empty subset of the time scale $\mathbb{T}$ including a fixed number
$t_{0}\in\mathbb{T}^{\ast}$ such that there exist operators $\delta_{\pm
}:[t_{0},\infty)_{\mathbb{T}}\times\mathbb{T}^{\ast}\rightarrow\mathbb{T}%
^{\ast}$ satisfying the following properties:

\begin{enumerate}
\item[P.1] The functions $\delta_{\pm}$ are strictly increasing with respect
to their second arguments, i.e., if
\[
(T_{0},t),(T_{0},u)\in\mathcal{D}_{\pm}:=\left\{  (s,t)\in\lbrack t_{0}%
,\infty)_{\mathbb{T}}\times\mathbb{T}^{\ast}:\delta_{\pm}(s,t)\in
\mathbb{T}^{\ast}\right\}  ,
\]
then
\[
T_{0}\leq t<u\text{ implies }\delta_{\pm}(T_{0},t)<\delta_{\pm}(T_{0},u),
\]

\item[P.2] If $(T_{1},u),(T_{2},u)\in\mathcal{D}_{-}$ with $T_{1}<T_{2}$, then%
\[
\delta_{-}(T_{1},u)>\delta_{-}(T_{2},u),
\]
and if $(T_{1},u),(T_{2},u)\in\mathcal{D}_{+}$ with $T_{1}<T_{2}$, then
\[
\delta_{+}(T_{1},u)<\delta_{+}(T_{2},u),
\]

\item[P.3] If $t\in\lbrack t_{0},\infty)_{\mathbb{T}}$, then $(t,t_{0}%
)\in\mathcal{D}_{+}$ and $\delta_{+}(t,t_{0})=t$. Moreover, if $t\in
\mathbb{T}^{\ast}$, then $(t_{0},t)$ $\in\mathcal{D}_{+}$ and $\delta
_{+}(t_{0},t)=t$ holds,

\item[P.4] If $(s,t)\in\mathcal{D}_{\pm}$, then $(s,\delta_{\pm}%
(s,t))\in\mathcal{D}_{\mp}$ and $\delta_{\mp}(s,\delta_{\pm}(s,t))=t$,

\item[P.5] If $(s,t)\in\mathcal{D}_{\pm}$ and $(u,\delta_{\pm}(s,t))\in
\mathcal{D}_{\mp}$, then $(s,\delta_{\mp}(u,t))\in\mathcal{D}_{\pm}$ and
\[
\delta_{\mp}(u,\delta_{\pm}(s,t))=\delta_{\pm}(s,\delta_{\mp}(u,t)).
\]

\end{enumerate}

\noindent Then the operators $\delta_{-}$ and $\delta_{+}$ associated with
$t_{0}\in\mathbb{T}^{\ast}$ (called the initial point) are said to be
\textit{backward and forward shift operators} on the set $\mathbb{T}^{\ast}$,
respectively. The variable $s\in\lbrack t_{0},\infty)_{\mathbb{T}}$ in
$\delta_{\pm}(s,t)$ is called the shift size. The values $\delta_{+}(s,t)$ and
$\delta_{-}(s,t)$ in $\mathbb{T}^{\ast}$ indicate $s$ units translation of the
term $t\in\mathbb{T}^{\ast}$ to the right and left, respectively. The sets
$\mathcal{D}_{\pm}$ are the domains of the shift operators $\delta_{\pm}$, respectively.
\end{definition}

\begin{example}
\cite{adrafdelay}Let $\mathbb{T=R}$ and $t_{0}=1$. The operators%
\begin{equation}
\delta_{-}(s,t)=\left\{
\begin{array}
[c]{cc}%
t/s & \text{if }t\geq0\\
st & \text{if }t<0
\end{array}
\right.  ,\ \ \ \text{for }s\in\lbrack1,\infty) \label{rs1}%
\end{equation}
and%
\begin{equation}
\delta_{+}(s,t)=\left\{
\begin{array}
[c]{cc}%
st & \text{if }t\geq0\\
t/s & \text{if }t<0
\end{array}
\right.  ,\ \ \ \text{for }s\in\lbrack1,\infty) \label{rs2}%
\end{equation}
are backward and forward shift operators (on the set $\mathbb{T}^{\ast
}=\mathbb{R-}\left\{  0\right\}  $) associated with the initial point
$t_{0}=1$. In the table below, we state different time scales with their
corresponding shift operators.
\[%
\begin{tabular}
[c]{|c||c|c|c|c|}\hline
$\mathbb{T}$ & $t_{0}$ & $\mathbb{T}^{\ast}$ & $\delta_{-}(s,t)$ & $\delta
_{+}(s,t)$\\\hline\hline
$\mathbb{R}$ & $0$ & $\mathbb{R}$ & $t-s$ & $t+s$\\\hline
$\mathbb{Z}$ & $0$ & $\mathbb{Z}$ & $t-s$ & $t+s$\\\hline
$q^{\mathbb{Z}}\cup\left\{  0\right\}  $ & $1$ & $q^{\mathbb{Z}}$ & $\frac
{t}{s}$ & $st$\\\hline
$\mathbb{N}^{1/2}$ & $0$ & $\mathbb{N}^{1/2}$ & $\sqrt{t^{2}-s^{2}}$ &
$\sqrt{t^{2}+s^{2}}$\\\hline
\end{tabular}
\ \ \ \ \ \ \ \ \ \ \
\]

\end{example}

The proof of the next lemma is a direct consequence of Definition \ref{shift}.

\begin{lemma}
\cite{adrafdelay}\label{lem pro} Let $\delta_{-}$ and $\delta_{+}$ be the
shift operators associated with the initial point $t_{0}$. We have

\begin{enumerate}
\item[i.] $\delta_{-}(t,t)=t_{0}$ for all $t\in\lbrack t_{0},\infty
)_{\mathbb{T}}$,

\item[ii.] $\delta_{-}(t_{0},t)=t$ for all $t\in\mathbb{T}^{\ast}$,

\item[iii.] If $(s,t)\in$ $\mathcal{D}_{+}$, then $\delta_{+}(s,t)=u$ implies
$\delta_{-}(s,u)=t$. Conversely, if $(s,u)\in$ $\mathcal{D}_{-}$, then
$\delta_{-}(s,u)=t$ implies $\delta_{+}(s,t)=u$.

\item[iv.] $\delta_{+}(t,\delta_{-}(s,t_{0}))=\delta_{-}(s,t)$ for all
$(s,t)\in$ $\mathcal{D}(\delta_{+})$ with $t\geq t_{0,}$

\item[v.] $\delta_{+}(u,t)=\delta_{+}(t,u)$ for all $(u,t)\in\left(  \lbrack
t_{0},\infty)_{\mathbb{T}}\times\lbrack t_{0},\infty)_{\mathbb{T}}\right)
\cap\mathcal{D}_{+}$,

\item[vi.] $\delta_{+}(s,t)\in\lbrack t_{0},\infty)_{\mathbb{T}}$ for all
$(s,t)\in$ $\mathcal{D}_{+}$ with $t\geq t_{0}$,

\item[vii.] $\delta_{-}(s,t)\in\lbrack t_{0},\infty)_{\mathbb{T}}$ for all
$(s,t)\in$ $\left(  [t_{0},\infty)_{\mathbb{T}}\times\lbrack s,\infty
)_{\mathbb{T}}\right)  \cap\mathcal{D}_{-}$,

\item[viii.] If $\delta_{+}(s,.)$ is $\Delta-$differentiable in its second
variable, then $\delta_{+}^{\Delta_{t}}(s,.)>0$,

\item[ix.] $\delta_{+}(\delta_{-}(u,s),\delta_{-}(s,v))=\delta_{-}(u,v)$ for
all $(s,v)\in\left(  \lbrack t_{0},\infty)_{\mathbb{T}}\times\lbrack
s,\infty)_{\mathbb{T}}\right)  \cap\mathcal{D}_{-}$ and $(u,s)\in\left(
\lbrack t_{0},\infty)_{\mathbb{T}}\times\lbrack u,\infty)_{\mathbb{T}}\right)
\cap\mathcal{D}_{-}$,

\item[x.] If $(s,t)\in\mathcal{D}_{-}$ and $\delta_{-}(s,t)=t_{0}$, then $s=t$.
\end{enumerate}
\end{lemma}

\subsection{Delay functions generated by shift operators}

Next, we define the delay function by means of shift operators on time scales.
Delay functions generated by shift operators were first introduced in
\cite{adrafdelay} to construct delay equations on time scales.

\begin{definition}
[Delay functions]\label{def 2}\cite{adrafdelay} Let $\mathbb{T}$ be a time
scale that is unbounded above and $\mathbb{T}^{\ast}$ an unbounded subset of
$\mathbb{T}$ including a fixed number $t_{0}\in\mathbb{T}^{\ast}$ such that
there exist shift operators $\delta_{\pm}:[t_{0},\infty)_{\mathbb{T}}%
\times\mathbb{T}^{\ast}\rightarrow\mathbb{T}^{\ast}$ associated with $t_{0}$.
Suppose that $h\in(t_{0},\infty)_{\mathbb{T}}$ is a constant such that
$(h,t)\in D_{\pm}$ for all $t\in\lbrack t_{0},\infty)_{\mathbb{T}}$, the
function $\delta_{-}(h,t)$ is differentiable with an $rd$-continuous
derivative $\delta_{-}^{\Delta_{t}}(h,t)$, and $\delta_{-}(h,t)$ maps
$[t_{0},\infty)_{\mathbb{T}}$ onto $[\delta_{-}(h,t_{0}),\infty)_{\mathbb{T}}%
$. Then the function $\delta_{-}(h,t)$ is called the delay function generated
by the shift $\delta_{-}$ on the time scale $\mathbb{T}$.
\end{definition}

It is obvious from P.2 in Definition \ref{def 2} and (ii) of Lemma
\ref{lem pro} that%
\begin{equation}
\delta_{-}(h,t)<\delta_{-}(t_{0},t)=t\text{ for all }t\in\lbrack t_{0}%
,\infty)_{\mathbb{T}}\text{.} \label{delay less}%
\end{equation}
Notice that $\delta_{-}(h,.)$ is strictly increasing and it is invertible.
Hence, by P.4-5%
\[
\delta_{-}^{-1}(h,t)=\delta_{+}(h,t).
\]

Hereafter, we shall suppose that $\mathbb{T}$ is a time scale with the delay
function $\delta_{-}(h,.):[t_{0},\infty)_{\mathbb{T}}\rightarrow\lbrack
\delta_{-}(h,t_{0}),\infty)_{\mathbb{T}}$, where $t_{0}\in\mathbb{T}$ is
fixed. Denote by $\mathbb{T}_{1}$ and $\mathbb{T}_{2}$ the sets%
\begin{equation}
\mathbb{T}_{1}=[t_{0},\infty)_{\mathbb{T}}\text{\ \ and }\mathbb{T}_{2}%
=\delta_{-}(h,\mathbb{T}_{1}). \label{T12}%
\end{equation}
Evidently, $\mathbb{T}_{1}$ is closed in $\mathbb{R}$. By definition we have
$\mathbb{T}_{2}=[\delta_{-}(h,t_{0}),\infty)_{\mathbb{T}}$. Hence,
$\mathbb{T}_{1}$ and $\mathbb{T}_{2}$ are both time scales. Let $\sigma_{1}$
and $\sigma_{2}$ denote the forward jumps on the time scales $\mathbb{T}_{1}$
and $\mathbb{T}_{2}$, respectively. By (\ref{delay less}-\ref{T12})
\[
\mathbb{T}_{1}\subset\mathbb{T}_{2}\subset\mathbb{T}.
\]
Thus,%
\[
\sigma(t)=\sigma_{2}(t)\text{ for all }t\in\mathbb{T}_{2}%
\]
and%
\[
\sigma(t)=\sigma_{1}(t)=\sigma_{2}(t)\text{ for all }t\in\mathbb{T}_{1}.
\]
That is, $\sigma_{1}$ and $\sigma_{2}$ are the restrictions of forward jump
operator $\sigma:\mathbb{T\rightarrow T}$ to the time scales $\mathbb{T}_{1}$
and $\mathbb{T}_{2}$, respectively, i.e.,%
\[
\sigma_{1}=\left.  \sigma\right\vert _{\mathbb{T}_{1}}\text{ and }\sigma
_{2}=\left.  \sigma\right\vert _{\mathbb{T}_{2}}\text{.}%
\]

\begin{lemma}
\cite{adrafdelay}The delay function $\delta_{-}(h,t)$ preserves the structure
of the points in $\mathbb{T}_{1}$. That is,%
\[
\sigma_{1}(\widehat{t})=\widehat{t}\text{ implies }\sigma_{2}(\delta
_{-}(h,\widehat{t}))=\delta_{-}(h,\widehat{t}).
\]%
\[
\sigma_{1}(\widehat{t})>\widehat{t}\text{ implies }\sigma_{2}(\delta
_{-}(h,\widehat{t})>\delta_{-}(h,\widehat{t}).
\]

\end{lemma}

Using the preceding lemma and applying the fact that $\sigma_{2}(u)=\sigma(u)$
for all $u\in\mathbb{T}_{2}$ we arrive at the following result.

\begin{corollary}
\cite{adrafdelay}\label{Cor 1} We have%
\[
\delta_{-}(h,\sigma_{1}(t))=\sigma_{2}(\delta_{-}(h,t))\text{ for all }%
t\in\mathbb{T}_{1}\text{.}%
\]
Thus,%
\begin{equation}
\delta_{-}(h,\sigma(t))=\sigma(\delta_{-}(h,t))\text{ for all }t\in
\mathbb{T}_{1}\text{.} \label{sigma delta}%
\end{equation}

\end{corollary}

By (\ref{sigma delta}) we have%
\[
\delta_{-}(h,\sigma(s))=\sigma(\delta_{-}(h,s))\text{ for all }s\in\lbrack
t_{0},\infty)_{\mathbb{T}}\text{.}%
\]
Substituting $s=\delta_{+}(h,t)$ we obtain%
\[
\delta_{-}(h,\sigma(\delta_{+}(h,t)))=\sigma(\delta_{-}(h,\delta
_{+}(h,t)))=\sigma(t)\text{.}%
\]
This and (iv) of Lemma \ref{lem pro} imply%
\[
\sigma(\delta_{+}(h,t))=\delta_{+}(h,\sigma(t))\text{ for all }t\in
\lbrack\delta_{-}(h,t_{0}),\infty)_{\mathbb{T}}\text{.}%
\]

\begin{example}
In the following, we give some time scales with their shift operators:
\[%
\begin{tabular}
[c]{|c||c|c|c|}\hline
$\mathbb{T}$ & $h$ & $\delta_{-}(h,t)$ & $\delta_{+}(h,t)$\\\hline\hline
$\mathbb{R}$ & $\in\mathbb{R}_{+}$ & $t-h$ & $t+h$\\\hline
$\mathbb{Z}$ & $\in\mathbb{Z}_{+}$ & $t-h$ & $t+h$\\\hline
$q^{\mathbb{Z}}\cup\left\{  0\right\}  $ & $\in q^{\mathbb{Z}_{+}}$ &
$\frac{t}{h}$ & $ht$\\\hline
$\mathbb{N}^{1/2}$ & $\in\mathbb{Z}_{+}$ & $\sqrt{t^{2}-h^{2}}$ & $\sqrt
{t^{2}+h^{2}}$\\\hline
\end{tabular}
\]

\end{example}

\begin{example}
There is no delay function $\delta_{-}(h,.):[0,\infty)_{\widetilde{\mathbb{T}%
}}\rightarrow\lbrack\delta_{-}(h,0),\infty)_{\mathbb{T}}$ on the time scale
$\widetilde{\mathbb{T}}\mathbb{=(-\infty},0]\cup\lbrack1,\infty)$.\newline
Suppose contrary that there exists a such delay function on $\widetilde
{\mathbb{T}}$. Then since $0$ is right scattered in $\widetilde{\mathbb{T}%
}_{1}:=[0,\infty)_{\widetilde{\mathbb{T}}}$ the point $\delta_{-}(h,0)$ must
be right scattered in $\widetilde{\mathbb{T}}_{2}=[\delta_{-}(h,0),\infty
)_{\mathbb{T}}$, i.e., $\sigma_{2}(\delta_{-}(h,0))>\delta_{-}(h,0)$. Since
$\sigma_{2}(t)=\sigma(t)$ for all $t\in\lbrack\delta_{-}(h,0),0)_{\mathbb{T}}%
$, we have%
\[
\sigma(\delta_{-}(h,0))=\sigma_{2}(\delta_{-}(h,0))>\delta_{-}(h,0).
\]
That is, $\delta_{-}(h,0)$ must be right scattered in $\widetilde{\mathbb{T}}%
$. However, in $\widetilde{\mathbb{T}}$ we have $\delta_{-}(h,0)<0$, that is,
$\delta_{-}(h,0)$ is right dense. This leads to a contradiction.
\end{example}

\section{Halanay type inequalities on time scales}

Let $\mathbb{T}$ be a time scale that is unbounded above\ and $t_{0}%
\in\mathbb{T}^{\ast}$ an element such that there exist the shift operators
$\delta_{\pm}:[t_{0},\infty)\times\mathbb{T}^{\ast}\rightarrow\mathbb{T}%
^{\ast}$ associated with $t_{0}$. Suppose that $h_{1}$, $h_{2}$,...$,h_{r}%
\in(t_{0},\infty)_{\mathbb{T}}$ are the constants with%
\[
t_{0}=h_{0}<h_{1}<h_{2}<...<h_{r}%
\]
and that there exist delay functions $\delta_{-}(h_{i},t)$, $i=1,2,...,r$, on
$\mathbb{T}$.

We define lower $\Delta$-derivative $\varphi^{\Delta_{-}}(t)$ of a function
$\varphi:\mathbb{T\rightarrow R}$ on time scales as follows:%
\begin{equation}
\varphi^{\Delta_{-}}(t)=\liminf\limits_{s\rightarrow t^{-}}\frac
{\varphi(s)-\varphi(\sigma(t))}{s-\sigma(t)}. \label{dini-}%
\end{equation}
Notice that%
\[
\varphi^{\Delta_{-}}(t)=\varphi^{\Delta}(t)
\]
provided that $\varphi$ is $\Delta$-differentiable at $t\in\mathbb{T}^{\kappa
}$.

Let $f(t,u,v)$ be a continuous function for all $\left(  u,v\right)  $ and
$t\in\lbrack t_{0},\alpha)_{\mathbb{T}}$. Hereafter, we suppose that $f$ is
monotone increasing with respect to $v$ and non-decreasing with respect to $u$.

\begin{proposition}
\label{pro1}Let $g(u_{1},u_{2},...,u_{r})$ be a continuous function that is
monotone increasing with respect to each of its arguments. If $\varphi$ and
$\psi$ are continuous functions satisfying%
\[
\varphi^{\Delta_{-}}(t)<f\left(  t,\varphi(t),g\left(  \varphi(\delta
_{-}(h_{1},t)),\varphi(\delta_{-}(h_{2},t)),...,\varphi(\delta_{-}%
(h_{r},t))\right)  \right)  ,
\]%
\[
\psi^{\Delta_{-}}(t)\geq f\left(  t,\psi(t),g\left(  \psi(\delta_{-}%
(h_{1},t)),\psi(\delta_{-}(h_{2},t)),...,\psi(\delta_{-}(h_{r},t))\right)
\right)  ,
\]
for all $t\in\lbrack t_{0},\alpha)_{\mathbb{T}}$ and $\varphi(s)<\psi(s)$ for
all $s\in\lbrack\delta_{-}(h_{r},t_{0}),t_{0}]_{\mathbb{T}}$, then%
\begin{equation}
\varphi(t)<\psi(t)\text{ for all }t\in(t_{0},\alpha)_{\mathbb{T}},
\label{ineq}%
\end{equation}
where $\alpha\in(t_{0},\infty)_{\mathbb{T}}$.
\end{proposition}

\begin{proof}
Suppose that (\ref{ineq}) does not hold for some $t\in(t_{0},\alpha
)_{\mathbb{T}}$. Then the set%
\[
M:=\left\{  t\in(t_{0},\alpha)_{\mathbb{T}}:\varphi(t)\geq\psi(t)\right\}  .
\]
is non-empty. Since $M$ is bounded below we can let $\xi:=\inf M$. If $\xi$ is
left scattered (i.e. $\sigma(\rho(\xi))=\xi$), then it follows from the
definition of $\xi~$that%
\[
\varphi(\rho(\xi))<\psi(\rho(\xi)),
\]%
\[
\varphi(\xi)\geq\psi(\xi).
\]
Since $\rho(\xi)$ is right scattered, the function $\varphi$ is $\Delta
$-differentiable at $\rho(\xi)$ (see \cite[Theorem 1.16, (ii)]{book}), and
hence, $\varphi^{\Delta_{-}}(\rho(\xi))=\varphi^{\Delta}(\rho(\xi))$.
Similarly we obtain $\psi^{\Delta_{-}}(\rho(\xi))=\psi^{\Delta}(\rho(\xi))$.
Thus,%
\begin{align*}
\varphi(\xi)  &  =\varphi(\sigma(\rho(\xi)))\\
&  =\varphi(\rho(\xi))+\mu(\rho(\xi))\varphi^{\Delta}(\rho(\xi))\\
&  =\varphi(\rho(\xi))+\mu(\rho(\xi))\varphi^{\Delta_{-}}(\rho(\xi))\\
&  <\varphi(\rho(\xi))\\
&  +\mu(\rho(\xi))f\left(  \rho(\xi),\varphi(\rho(\xi)),g\left(
\varphi(\delta_{-}(h_{1},\rho(\xi))),\varphi(\delta_{-}(h_{2},\rho
(\xi))),...,\varphi(\delta_{-}(h_{r},\rho(\xi)))\right)  \right) \\
&  <\psi(\rho(\xi))\\
&  +\mu(\rho(\xi))f\left(  \rho(\xi),\psi(\rho(\xi)),g\left(  \psi(\delta
_{-}(h_{1},\rho(\xi))),\psi(\delta_{-}(h_{2},\rho(\xi))),...,\psi(\delta
_{-}(h_{r},\rho(\xi)))\right)  \right) \\
&  \leq\psi(\rho(\xi))+\mu(\rho(\xi))\psi^{\Delta_{-}}(\rho(\xi))\\
&  =\psi(\rho(\xi))+\mu(\rho(\xi))\psi^{\Delta}(\rho(\xi))\\
&  =\psi(\sigma(\rho(\xi)))\\
&  =\psi(\xi).
\end{align*}
This leads to a contradiction. If $\xi$ is left dense, then we have $\xi
>t_{0}$ and%
\[
\varphi(\xi)=\psi(\xi).
\]
Since%
\[
\delta_{-}(h_{r},\xi)<\xi\text{ for all }i=1,2,...,r
\]
and%
\[
\varphi(s)<\psi(s)\text{ for all }s\in\lbrack\delta_{-}(h_{r},\xi
),\xi)_{\mathbb{T}},
\]
we obtain
\[
g(\varphi(\delta_{-}(h_{1},\xi)),...,\varphi(\delta_{-}(h_{r},\xi)))\leq
g(\psi(\delta_{-}(h_{1},\xi)),...,\psi(\delta_{-}(h_{r},\xi))),
\]
and therefore,%
\begin{align*}
\varphi^{\Delta_{-}}(\xi)  &  <f\left(  \xi,\varphi(\xi),g\left(
\varphi(\delta_{-}(h_{1},\xi)),\varphi(\delta_{-}(h_{2},\xi)),...,\varphi
(\delta_{-}(h_{r},\xi))\right)  \right) \\
&  \leq f\left(  \xi,\psi(\xi),g\left(  \psi(\delta_{-}(h_{1},\xi
)),\psi(\delta_{-}(h_{2},\xi)),...,\psi(\delta_{-}(h_{r},\xi))\right)  \right)
\\
&  \leq\psi^{\Delta_{-}}(\xi).
\end{align*}
On the other hand, since%
\[
\frac{\varphi(s)-\varphi(\sigma(\xi))}{s-\sigma(\xi)}\geq\frac{\psi
(s)-\psi(\sigma(\xi))}{s-\sigma(\xi)}%
\]
for all $s\in\lbrack\delta_{-}(h_{r},\xi),\xi)_{\mathbb{T}}$ we get by
(\ref{dini-}) that%
\[
\varphi^{\Delta_{-}}(\xi)\geq\psi^{\Delta_{-}}(\xi).
\]
This also leads to a contradiction and so this completes the proof.
\end{proof}

\begin{proposition}
\label{pro2}If%
\[
\omega^{\Delta}(t)\leq f\left(  t,\omega(t),g\left(  \omega(\delta_{-}%
(h_{1},t)),\omega(\delta_{-}(h_{2},t)),...,\omega(\delta_{-}(h_{r},t))\right)
\right)
\]
for $t\in\lbrack s_{0},\delta_{+}(\alpha,s_{0}))_{\mathbb{T}}$ and
$y(t;s_{0},\omega)$ is a solution of the equation%
\[
y^{\Delta}(t)=f\left(  t,y(t),g\left(  y(\delta_{-}(h_{1},t)),y(\delta
_{-}(h_{2},t)),...,y(\delta_{-}(h_{r},t))\right)  \right)  ,
\]
which coincides with $\omega$ in $[\delta_{-}(h_{r},s_{0}),s_{0}]_{\mathbb{T}%
}$, then, supposing that this solution is defined in $[s_{0},\delta_{+}%
(\alpha,s_{0}))_{\mathbb{T}}$, it follows that $\omega(t)\leq y(t;s_{0}%
,\omega)$ for $t\in\lbrack s_{0},\delta_{+}(\alpha,s_{0}))_{\mathbb{T}}$.
\end{proposition}

\begin{proof}
Let $\varepsilon_{n}$ be a sequence of positive numbers tending monotonically
to zero, and $y_{n}$ be a solution of the equation%
\[
y^{\Delta}(t)=f\left(  t,y(t),g\left(  y(\delta_{-}(h_{1},t)),y(\delta
_{-}(h_{2},t)),...,y(\delta_{-}(h_{r},t))\right)  \right)  +\varepsilon_{n},
\]
which in $[\delta_{-}(h_{r},s_{0}),s_{0}]_{\mathbb{T}}$ coincides with
$\omega+\varepsilon_{n}$. On the basis of the preceding proposition, we have%
\[
y_{n+1}(t)<y_{n}(t)\text{ }%
\]
and%
\[
\lim_{n\rightarrow\infty}y_{n}(t)=y(t;s_{0},\omega)
\]
for all $t\in\lbrack s_{0},\delta_{+}(\alpha,s_{0}))_{\mathbb{T}}$. On the
basis of Proposition \ref{pro1} we have $\omega(t)<y_{n}(t)$ for $t\in\lbrack
s_{0},\delta_{+}(\alpha,s_{0}))_{\mathbb{T}}$, and hence, $\omega(t)\leq
y(t;s_{0},\omega)$. The proof is complete.
\end{proof}

Hereafter, we will denote by $\widetilde{\mu}$ the function defined by%
\begin{equation}
\widetilde{\mu}(t):=\sup\limits_{s\in\lbrack\delta_{-}(h_{r},t_{0}%
),t]_{\mathbb{T}}}\mu(s) \label{mutilda}%
\end{equation}
for $t\in\lbrack t_{0},\infty)_{\mathbb{T}}$. It is obvious that the sets
$\mathbb{R}$, $\mathbb{Z}$, $\overline{q^{\mathbb{Z}}}=\{q^{n}:n\in\mathbb{Z}$
and $q>1\}\cup\left\{  0\right\}  $, $h\mathbb{Z=\{}hn:n\in\mathbb{Z}$ and
$h>0\mathbb{\}}$ are the examples of time scales on which $\widetilde{\mu}%
=\mu$.

\begin{theorem}
\label{thm1}Let $x$ be a function satisfying the inequality%
\begin{equation}
x^{\Delta}(t)\leq-p(t)x(t)+\sum_{i=0}^{r}q_{i}(t)x^{\ell}(\delta_{-}%
(h_{i},t)),\ \ t\in\lbrack t_{0},\infty)_{\mathbb{T}}\text{,} \label{Ineq1}%
\end{equation}
where $\ell\in(0,1]$ is a constant, $p$ and $q_{i}$, $i=0,1,...,r$, are
continuous and bounded functions satisfying $1-\widetilde{\mu}(t)p(t)\geq0$;
$q_{i}(t)\geq0$, $i=0,1,...,r-1$; $q_{r}(t)>0$ for all $t\in\lbrack
t_{0},\infty)_{\mathbb{T}}$. Suppose that%
\begin{equation}
p(t)-\sum_{i=0}^{r}q_{i}(t)>0\text{ \ for all }t\in\lbrack t_{0}%
,\infty)_{\mathbb{T}}. \label{Cond1}%
\end{equation}
Then there exist a positively regressive function $\lambda:[t_{0}%
,\infty)_{\mathbb{T}}\mathbb{\rightarrow(-\infty},0)$ and $K_{0}>1$ such that%
\begin{equation}
x(t)\leq K_{0}e_{\lambda}(t,t_{0})\text{ for }t\in\lbrack t_{0},\infty
)_{\mathbb{T}} \label{ineq2}%
\end{equation}

\end{theorem}

\begin{proof}
Consider the delay dynamic equation%
\begin{equation}
y^{\Delta}(t)=-p(t)y(t)+\sum_{i=0}^{r}q_{i}(t)y^{\ell}(\delta_{-}%
(h_{i},t)),\ \ t\in\lbrack t_{0},\infty)_{\mathbb{T}} \label{Eq.1}%
\end{equation}
We look for a solution of equation (\ref{Eq.1}) in the form $e_{\lambda
}(t,t_{0})$, where $\lambda:\mathbb{T\rightarrow(-\infty},0)$ is positively
regressive (i.e. $1+\mu(t)\lambda(t)>0$) and rd-continuous. First note that%
\[
e_{\lambda}^{\Delta}(t,t_{0})=\lambda(t)e_{\lambda}(t,t_{0}).
\]
For a given $K>1$, the function $Ke_{\lambda}(t,t_{0})$ is a solution of
(\ref{Eq.1}) if and only if $\lambda$ is a root of the characteristic
polynomial $P(t,\lambda)$ defined by%
\begin{align}
P(t,\lambda)  &  :=\left(  \lambda+p(t)\right)  e_{\lambda}(t,\delta_{-}%
(h_{r},t))e_{\lambda}^{1-\ell}(\delta_{-}(h_{r},t),t_{0})\nonumber\\
&  -K^{\ell-1}\sum_{i=0}^{r}q_{i}(t)e_{\lambda}^{\ell}(\delta_{-}%
(h_{i},t),\delta_{-}(h_{r},t)). \label{cep1}%
\end{align}
For each fixed $t\in\lbrack t_{0},\infty)_{\mathbb{T}}$ define the set%
\begin{equation}
S(t):=\left\{  k\in(-\infty,0):1+\widetilde{\mu}(t)k>0\right\}  .
\label{set s(t)}%
\end{equation}
It follows from Lemma \ref{remark osc} that if $k$ is a scalar in $S(t)$, then
$0<1+\widetilde{\mu}(t)k\leq1+\mu(u)k$ for all $u\in$ $\left[  \delta
_{-}(h_{r},t_{0}),t\right]  _{\mathbb{T}}$ and
\begin{equation}
0<e_{k}(\tau,s)\leq\exp(k(\tau-s)), \label{exp2}%
\end{equation}
for all $\tau\in$ $\left[  \delta_{-}(h_{r},t_{0}),t\right]  _{\mathbb{T}}$
with $\tau\geq s$. It is obvious from (\ref{exp}) and (\ref{cylinder}) that
for each fixed $t\in\lbrack t_{0},\infty)_{\mathbb{T}}$ the function $P(t,k)$
is continuous with respect to $k$ in $S(t)$. Since $e_{0}(t,t_{0})=1$ we have%
\begin{equation}
P(t,0)=p(t)-K^{\ell-1}\sum_{i=0}^{r}q_{i}(t)>0. \label{P1}%
\end{equation}
\newline Let $t\in\lbrack t_{0},\infty)_{\mathbb{T}}$ be fixed. If the
interval $\left[  \delta_{-}(h_{r},t_{0}),t\right]  _{\mathbb{T}}$ has no any
right scattered points, then $\widetilde{\mu}(t)=0$ and $S(t)=(-\infty,0)$. By
(\ref{exp2}), we get%
\[
\lim\limits_{k\rightarrow-\infty}e_{k}(t,s)=0,
\]
and hence,%
\[
\lim\limits_{k\rightarrow-\infty}P(t,k)=-K^{\ell-1}q_{r}(t)<0.
\]
If the interval $\left[  \delta_{-}(h_{r},t_{0}),t\right]  _{\mathbb{T}}$ has
some right scattered points (i.e. if $\widetilde{\mu}(t)>0$), then we have
$S(t)=(-\frac{1}{\widetilde{\mu}(t)},0)$. For all $k\in(-\frac{1}%
{\widetilde{\mu}(t)},0)$ we have $e_{k}(t,s)>0$. Since $1-\widetilde{\mu
}(t)p(t)\geq0$ for all $t\in\lbrack t_{0},\infty)_{\mathbb{T}}$, we obtain%
\begin{align*}
\lim\limits_{k\rightarrow-\frac{1}{\widetilde{\mu}(t)}^{+}}P(t,k)  &  =\left(
-\frac{1}{\widetilde{\mu}(t)}+p(t)\right)  \lim\limits_{k\rightarrow-\frac
{1}{\widetilde{\mu}(t)}^{+}}\left[  e_{k}(t,\delta_{-}(h_{r},t))e_{k}^{1-\ell
}(\delta_{-}(h_{r},t),t_{0})\right] \\
&  -K^{\ell-1}\sum_{i=0}^{r-1}q_{i}(t)\lim\limits_{k\rightarrow-\frac
{1}{\widetilde{\mu}(t)}^{+}}e_{k}^{\ell}(\delta_{-}(h_{i},t),\delta_{-}%
(h_{r},t))-K^{\ell-1}q_{r}(t)\\
&  <-K^{\ell-1}q_{r}(t)<0.
\end{align*}
Therefore, for each fixed $t\in\lbrack t_{0},\infty)_{\mathbb{T}}$, we obtain%
\begin{equation}
0>-K^{\ell-1}q_{r}(t)\geq\left\{
\begin{array}
[c]{cc}%
\lim\limits_{k\rightarrow-\frac{1}{\widetilde{\mu}(t)}^{+}}P(t,k) & \text{if
}\widetilde{\mu}(t)>0\\
\lim\limits_{k\rightarrow-\infty}P(t,k) & \text{if }\widetilde{\mu}(t)=0
\end{array}
\right.  . \label{P2}%
\end{equation}
It follows from the continuity of $P$ in $k$ and (\ref{P1}-\ref{P2}) that for
each fixed $t\in\lbrack t_{0},\infty)_{\mathbb{T}}$, there exists a largest
element $k_{0}$ of the set $S(t)$ such that%
\[
P(t,k_{0})=0.
\]
Using all these largest elements we can construct a positively regressive
function $\lambda:[\delta_{-}(h_{r},t_{0}),\infty)_{\mathbb{T}}%
\mathbb{\rightarrow(-\infty},0)$ by%
\begin{equation}
\lambda(t):=\max\left\{  k\in S(t):P(t,k)=0\right\}  \label{lambda}%
\end{equation}
so that for a given $K>1$, $y(t)=Ke_{\lambda}(t,t_{0})$ is a solution to
(\ref{Eq.1}). \newline If $y(t)$ be a solution of (\ref{Eq.1}), $x(t)$
satisfies (\ref{Ineq1}), and $x(t)\leq y(t)$ for all $t\in\lbrack\delta
_{-}(h_{r},t_{0}),t_{0}]_{\mathbb{T}}$, then by Proposition \ref{pro2} the
inequality $x(t)\leq y(t)$ holds for all $t\in\lbrack t_{0},\infty
)_{\mathbb{T}}$. For a given $K>1$, we have%
\[
\inf\limits_{t\in\lbrack\delta_{-}(h_{r},t_{0}),t_{0}]_{\mathbb{T}}%
}Ke_{\lambda}(t,t_{0})=K,
\]
hence, by choosing a $K_{0}>1$ with%
\[
K_{0}>\sup\limits_{t\in\lbrack\delta_{-}(h_{r},t_{0}),t_{0}]_{\mathbb{T}}%
}x(t),
\]
we get%
\[
x(t)<K_{0}e_{\lambda}(t,t_{0})\text{ for all }t\in\lbrack\delta_{-}%
(h_{r},t_{0}),t_{0}]_{\mathbb{T}}.
\]
It follows on the basis of Proposition \ref{pro2} that the inequality
\[
x(t)\leq K_{0}e_{\lambda}(t,t_{0})
\]
holds for all $t\in\lbrack t_{0},\infty)_{\mathbb{T}}$. This completes the proof.
\end{proof}

In next two examples, we apply Theorem \ref{thm1} to the time scales
$\mathbb{T}=\mathbb{Z}$ and $\mathbb{T}=q^{\mathbb{N}}$ to derive some results
for difference and $q-$difference inequalities.

\begin{example}
\label{Ex1}Let $\mathbb{T=Z}$; $t_{0}=0$; $\delta_{-}(h_{i},t)=t-h_{i}$,
$h_{i}\in\mathbb{N}$, $i=1,2,...,r-1$; $h_{r}\in\mathbb{Z}^{+}$, and
$0=h_{0}<h_{1}<...<h_{r}$. \ Assume that $p$ and $q_{i}\geq0$, $i=0,1,2,...,r$%
, are the scalars satisfying $q_{r}>0$ and
\[
\sum_{i=0}^{r}q_{i}<p\leq1.
\]
Then the equation (\ref{Eq.1}) becomes%
\begin{equation}
\Delta y(t)=-py(t)+\sum_{i=0}^{r}q_{i}y^{\ell}(t-h_{i}),\ \ t\in\left\{
0,1,...\right\}  . \label{eqdiif}%
\end{equation}
The characteristic polynomial and the set $S(t)$ given by (\ref{cep1}) and
(\ref{set s(t)}) turn into%
\[
P(t,\lambda)=(\lambda+p)(1+\lambda)^{h_{r}}(1+\lambda)^{\left(  1-\ell\right)
\left(  t-h_{r}\right)  }-K^{\ell-1}\sum_{i=0}^{r}q_{i}(1+\lambda
)^{\ell\left(  h_{r}-h_{i}\right)  }%
\]
and%
\[
S(t)=(-1,0)\text{ for all }t\in\left\{  0,1,...\right\}  ,
\]
respectively. Let $\left\{  x(t)\right\}  $, $t\in\lbrack-h_{r},\infty
)_{\mathbb{Z}}$ be a sequence satisfying the inequality
\[
\Delta x(t)\leq-px(t)+\sum_{i=0}^{r}q_{i}x^{\ell}(t-h_{i}),\ \ t\in\left\{
0,1,...\right\}  .
\]
Then by Theorem \ref{thm1}, we conclude that there exists a constant $K_{0}>1$
such that%
\[
x(t)<K_{0}\prod_{s=0}^{t-1}(1+\lambda_{0}(s))\text{, }t\in\left\{
0,1,...\right\}  ,
\]
where $\lambda_{0}:\mathbb{Z\rightarrow}(-1,0)$ is a positively regressive
function defined by%
\begin{align*}
\lambda_{0}(t)  &  =\max\left\{  \nu\in(-1,0):\left(  \nu+p\right)
(1+\nu)^{h_{r}}(1+\nu)^{\left(  1-\ell\right)  \left(  t-h_{r}\right)
}\right. \\
&  -\left.  K^{\ell-1}\sum_{i=0}^{r}q_{i}(1+\nu)^{\ell\left(  h_{r}%
-h_{i}\right)  }=0\right\}  .
\end{align*}

\end{example}

\begin{remark}
\label{rem imp1}Example \ref{Ex1} shows that the result in \cite[Theorem
2.1]{udpin} and \cite[Theorem 2.1]{agarwal} are the particular cases of
Theorem \ref{thm1} when $\mathbb{T=Z}$. Moreover, unlike the ones in
\cite[Theorem 2.1]{udpin} and \cite[Theorem 2.1]{agarwal}, the coefficients
$p$\ and $q_{i}$, $i=0,1,...,r$, of the dynamic inequality considered in
Theorem \ref{thm1} are allowed to depend on the parameter $t$. Hence, even for
the particular case $\mathbb{T=Z}$ we have a more general result.
\end{remark}

\begin{example}
Let $\mathbb{T=}q^{\mathbb{N}}:=\left\{  q^{n}:n\in\mathbb{N}\text{ and
}q>1\right\}  $, $t_{0}=1$, $\delta_{-}(h_{i},t)=t/h_{i}$, where $h_{i}\in
q^{\mathbb{N}}$, $1=h_{0}<h_{1}<...<h_{r}$. Let $x$ be a function satisfying
the inequality%
\[
D_{q}x(t)\leq-p(t)x(t)+\sum_{i=0}^{r}\zeta_{i}(t)x^{\ell}(\frac{t}{h_{i}%
}),\ \ t\in q^{\mathbb{N}}\text{,}%
\]
where $D_{q}x(t)$ is defined as in (\ref{q derivative}). Let%
\[
q^{\mathbb{Z}}:=\left\{  q^{n}:n\in\mathbb{Z}\text{ and }q>1\right\}  .
\]
Suppose that $p$ and $\zeta_{i}$, $i=0,1,...,r$, are continuous and bounded
functions satisfying $1-p(t)(q-1)t\geq0$; $\zeta_{i}(t)\geq0$, $i=1,...,r-1$;
$\zeta_{r}(t)>0$, and%
\[
p(t)-\sum_{i=0}^{r}\zeta_{i}(t)>0
\]
for all $t\in\lbrack1,\infty)\cap q^{\mathbb{Z}}$. Then there exists a
constant $K_{0}>1$ such that%
\[
x(t)\leq K_{0}%
{\displaystyle\prod\limits_{s\in\lbrack1,t)\cap q^{\mathbb{N}}}}
\left(  1+\lambda(s)(q-1)s\right)  \text{ for all }t\in q^{\mathbb{N}},
\]
in which $\lambda$ denotes the function defined by%
\[
\lambda(t):=\max\left\{  k\in\left(  -1/(q-1)t,0\right)  :P(t,k)=0\right\}
,\text{ \ \ }t\in q^{\mathbb{N}}%
\]
where%
\begin{align*}
P(t,k)  &  :=\left(  k+p(t)\right)
{\displaystyle\prod\limits_{s\in\lbrack\frac{t}{h_{r}},t)\cap q^{\mathbb{Z}}}}
\left(  1+k(q-1)s\right) \\
&  \times A(t,k)\\
&  -K^{\ell-1}\sum_{i=0}^{r}\zeta_{i}(t)%
{\displaystyle\prod\limits_{s\in\lbrack\frac{t}{h_{r}},\frac{t}{h_{i}})\cap
q^{\mathbb{Z}}}}
\left(  1+k(q-1)s\right)  ^{\ell},
\end{align*}
where%
\[
A(t,k)=\left\{
\begin{array}
[c]{cc}%
{\displaystyle\prod\limits_{s\in\lbrack1,\frac{t}{h_{r}})\cap q^{\mathbb{Z}}}}
\left(  1+k(q-1)s\right)  ^{1-\ell} & \text{if }t>h_{r}\\%
{\displaystyle\prod\limits_{s\in\lbrack\frac{t}{h_{r}},1)\cap q^{\mathbb{Z}}}}
\left(  1+k(q-1)s\right)  ^{\ell-1} & \text{if }t<h_{r}%
\end{array}
\right.  .
\]

\end{example}

\begin{theorem}
\label{thm new halanay}Let $\tau\in\lbrack t_{0},\infty)_{\mathbb{T}}$ be a
constant such that there exists a delay function $\delta_{-}(\tau,t)$ on
$\mathbb{T}$. Let $x$ be a function satisfying the inequality%
\[
x^{\Delta}(t)\leq-p(t)x(t)+q(t)\sup\limits_{s\in\left[  \delta_{-}%
(\tau,t),t\right]  }x^{\ell}(s),\ \ t\in\lbrack t_{0},\infty)_{\mathbb{T}%
}\text{,}%
\]
where $\ell\in(0,1]$ is a constant. Suppose that $p$ and $q$ are the
continuous and bounded functions satisfying $p(t)>q(t)>0$ and $1-\widetilde
{\mu}(t)p(t)\geq0$ for all $t\in\lbrack t_{0},\infty)_{\mathbb{T}}$. Then
there exists a constant $M_{0}>0$ such that every solution $x$ to Eq.
(\ref{3.1}) satisfies
\[
x(t)\leq M_{0}e_{\widetilde{\lambda}}(t,t_{0})\text{,}%
\]
where $\widetilde{\lambda}$ is a positively regressive function chosen as in
(\ref{lambda tilda}).
\end{theorem}

\begin{proof}
We proceed as we did in the proof of Theorem \ref{thm1}. Consider the dynamic
equation%
\begin{equation}
y^{\Delta}(t)=-py(t)+q\sup_{s\in\left[  \delta_{-}(\tau,t),t\right]  }y^{\ell
}(s),\ \ t\in\lbrack t_{0},\infty)_{\mathbb{T}}\text{.} \label{eqsup}%
\end{equation}
For a given $M>1$, $Me_{\lambda}(t,t_{0})$ is a solution of (\ref{Eq.1}) if
and only if $\lambda$ is a root of the characteristic polynomial
$\widetilde{P}(t,\lambda)$ defined by%
\[
\widetilde{P}(t,\lambda):=\left(  \lambda+p(t)\right)  e_{\lambda}%
(t,t_{0})-M^{\ell}q(t)\sup_{s\in\left[  \delta_{-}(\tau,t),t\right]
}e_{\lambda}^{\ell}(s,t_{0}).
\]
For each fixed $t\in\lbrack t_{0},\infty)_{\mathbb{T}}$ define the set%
\[
S(t):=\left\{  k\in(-\infty,0):1+\widetilde{\mu}(t)k>0\right\}  .
\]
It is obvious that for each fixed $t\in\lbrack t_{0},\infty)_{\mathbb{T}}$ and
for all $k\in S(t)$ we have%
\[
\widetilde{P}(t,k)=\left(  k+p(t)\right)  e_{k}(t,t_{0})-M^{\ell}%
q(t)e_{k}^{\ell}(\delta_{-}(\tau,t),t_{0})\text{. }%
\]
As we did in the proof of Theorem \ref{thm1}, one may easily show that for
each $t\in\lbrack t_{0},\infty)_{\mathbb{T}}$, there exists a largest element
of $S(t)$ such that $P(t,k)=0$. Using these largest elements we can define a
positively regressive function $\widetilde{\lambda}:[\delta_{-}(h_{r}%
,t_{0}),\infty)_{\mathbb{T}}\mathbb{\rightarrow(-\infty},0)$ by%
\begin{equation}
\widetilde{\lambda}(t):=\max\left\{  k\in S(t):\widetilde{P}(t,k)=0\right\}
\label{lambda tilda}%
\end{equation}
so that for a given $M>1$ $y(t)=Me_{\widetilde{\lambda}}(t,t_{0})$ is a
solution to (\ref{eqsup}). The rest of the proof can be done similar to that
of Theorem \ref{thm1}.
\end{proof}

\begin{remark}
\label{rem imp 2}Notice that Theorem \ref{thm new halanay} gives Lemma
\ref{lem halanay} in the particular case when $\mathbb{T=R}$ and $\ell=1$.
Moreover, since there is no nonnegativity condition on the function $x$,
Theorem \ref{thm new halanay} provides not only a generalization but also a
relaxation of Theorem \ref{gopalsamy}. Similar, relaxation is valid also for
the discrete case (see \cite[Theorem 2.1]{liz}).
\end{remark}

We finalize this section by giving a result for functions satisfying the
dynamic inequality%
\begin{equation}
x^{\Delta}(t)\leq-p(t)x(t)+%
{\displaystyle\prod\limits_{i=0}^{r}}
\beta_{i}(t)x^{\alpha_{i}}(\delta_{-}(h_{i},t)), \label{newineq}%
\end{equation}
where $\alpha_{i}\in\mathbb{(}0,\infty)$, $i=0,1,...,r$, are the scalars with
$\sum_{i=0}^{r}\alpha_{i}=1$. Let the characteristic polynomial $Q(t,k)$ and
the set $S(t)$ be defined by
\[
Q(t,k):=\left(  \lambda+p\right)  e_{\lambda}(t,t_{0})-%
{\displaystyle\prod\limits_{i=0}^{r}}
\beta_{i}e_{\lambda}^{\alpha_{i}}(\delta_{-}(h_{i},t),t_{0})
\]
and (\ref{set s(t)}), respectively. Applying the similar procedure to that
used in the proof of Theorem \ref{thm1} we arrive at the next result.

\begin{theorem}
\label{thm2}Let $x$ be a $\Delta-$differentiable function satisfying
(\ref{newineq}), where $\alpha_{i}\in\mathbb{(}0,\infty)$, $i=0,1,...,r$, are
the scalars with $\sum_{i=0}^{r}\alpha_{i}=1$; $p$ and $\beta_{i}$,
$i=0,1,...,r$, are continuous functions with the property that $1-\widetilde
{\mu}(t)p(t)\geq0$; $\beta_{i}(t)>0$, $i=0,1,...,r$, for all $t\in\lbrack
t_{0},\infty)_{\mathbb{T}}$. Suppose that
\[
p(t)-%
{\displaystyle\prod\limits_{i=0}^{r}}
\beta_{i}(t)>0
\]
for all $t\in\lbrack t_{0},\infty)_{\mathbb{T}}$. Then there exists a constant
$L_{0}>0$ such that%
\[
x(t)\leq L_{0}e_{\gamma}(t,t_{0})\text{ for }t\in\lbrack t_{0},\infty
)_{\mathbb{T}},
\]
where $\gamma:[t_{0},\infty)_{\mathbb{T}}\mathbb{\rightarrow(-\infty},0)$ is a
positively regressive function given by%
\begin{equation}
\gamma(t):=\max\left\{  k\in S(t):Q(t,k)=0\right\}  . \label{gamma}%
\end{equation}

\end{theorem}

Note that Theorem \ref{thm2} gives \cite[Theorem 2.2]{udpin} in the particular
case when $\mathbb{T=Z}$, $t_{0}=0$, $\delta_{-}(h_{i},t)=t-h_{i}$,
$i=0,1,2,..r$.

\section{Global stability of nonlinear dynamic equations}

In this section, by means of Halanay type inequalities we gave in the previous
section, we propose some sufficient conditions guaranteeing global stability
of nonlinear dynamic equations in the form%
\begin{equation}
x^{\Delta}(t)=-p(t)x(t)+F(t,x(t),x(\delta_{-}(h_{1},t)),...,x(\delta_{-}%
(h_{r},t))) \label{3.1}%
\end{equation}
for $t\in\lbrack t_{0},\infty)_{\mathbb{T}}$.

\begin{theorem}
\label{thm3.1}Let $p$\ and $q_{i}$, $i=0,1,...,r$, be continuous and bounded
functions satisfying $1-\widetilde{\mu}(t)p(t)>0$; $q_{i}(t)\geq0$,
$i=0,1,...,r$; $q_{r}(t)>0$ and%
\[
p(t)-\sum_{i=0}^{r}q_{i}(t)>0
\]
for all $t\in\lbrack t_{0},\infty)_{\mathbb{T}}$.\ Let $\ell\in(0,1]$ be a
constant. Assume that there exist scalars $h_{i}\in\lbrack t_{0}%
,\infty)_{\mathbb{T}}$, $i=0,1,...,r$, such that $h_{0}=t_{0}$, $\delta
_{-}(h_{i},t)$, $i=1,...,r$, are delay functions on $\mathbb{T}$, and
\begin{equation}
\left\vert F(t,x(t),x(\delta_{-}(h_{1},t)),...,x(\delta_{-}(h_{r}%
,t)))\right\vert \leq\sum_{i=0}^{r}q_{i}(t)\left\vert x(\delta_{-}%
(h_{i},t))\right\vert ^{\ell} \label{3.2}%
\end{equation}
for all $(t,x(t),x(\delta_{-}(h_{1},t)),...,x(\delta_{-}(h_{r},t)))\in\lbrack
t_{0},\infty)_{\mathbb{T}}\times\mathbb{R}^{r+1}$. Then there exists a
constant $M_{0}>1$ such that every solution $x$ to Eq. (\ref{3.1}) satisfies
\[
\left\vert x(t)\right\vert \leq M_{0}e_{\lambda}(t,t_{0})\text{,}%
\]
where $\lambda$ is a positively regressive function chosen as in (\ref{lambda}).
\end{theorem}

\begin{proof}
Let%
\[
\xi:=\ominus(-p)=\frac{p}{1-\mu p}.
\]
Multiplying both sides of Eq. (\ref{3.1}) by $e_{\xi}(t,t_{0})$ and
integrating the resulting equation from $t_{0}$ to $t$ we get that%
\begin{equation}
x(t)=x_{0}e_{\ominus\xi}(t,t_{0})+\int_{t_{0}}^{t}F(s,x(s),x(\delta_{-}%
(h_{1},s)),...,x(\delta_{-}(h_{r},s)))e_{\ominus\xi}(t,\sigma(s))\Delta s.
\label{new form of 3.1}%
\end{equation}
It is straightforward to show that a solution $x(t)$ to Eq.
(\ref{new form of 3.1}) satisfies (\ref{3.1}). This means every solution of
Eq. (\ref{3.1}) can be rewritten in the form of (\ref{new form of 3.1}). By
using (\ref{3.2}) we obtain%
\[
\left\vert x(t)\right\vert \leq\left\vert x_{0}\right\vert e_{\ominus\xi
}(t,t_{0})+\int_{t_{0}}^{t}\sum_{i=0}^{r}q_{i}(s)\left\vert x(\delta_{-}%
(h_{i},s))\right\vert ^{\ell}e_{\ominus\xi}(t,\sigma(s))\Delta s.
\]
Let the function $y$ be defined as follows:%
\[
y(t)=\left\vert x(t)\right\vert \text{ for }t\in\lbrack\delta_{-}(h_{r}%
,t_{0}),t_{0}]_{\mathbb{T}}%
\]
and%
\[
y(t)=\left\vert x_{0}\right\vert e_{\ominus\xi}(t,t_{0})+\int_{t_{0}}^{t}%
\sum_{i=0}^{r}q_{i}(s)\left\vert x(\delta_{-}(h_{i},s))\right\vert ^{\ell
}e_{\ominus\xi}(t,\sigma(s))\Delta s\text{ for }[t_{0},\infty)_{\mathbb{T}}.
\]
Then we have $\left\vert x(t)\right\vert \leq y(t)$ for all $t\in\lbrack
\delta_{-}(h_{r},t_{0}),\infty)_{\mathbb{T}}$. By \cite[Theorem 1.117]{book}
we get that%
\begin{align*}
y^{\Delta}(t)  &  =-p(t)\left(  \left\vert x_{0}\right\vert e_{\ominus\xi
}(t,t_{0})+\int_{t_{0}}^{t}\sum_{i=0}^{r}q_{i}(s)\left\vert x(\delta_{-}%
(h_{i},s))\right\vert ^{\ell}e_{\ominus\xi}(t,\sigma(s))\Delta s\right) \\
&  +\sum_{i=0}^{r}q_{i}(t)\left\vert x(\delta_{-}(h_{i},t))\right\vert ^{\ell
}\\
&  =-p(t)y(t)+\sum_{i=0}^{r}q_{i}(t)\left\vert x(\delta_{-}(h_{i}%
,t))\right\vert ^{\ell}\\
&  \leq-p(t)y(t)+\sum_{i=0}^{r}q_{i}(t)y^{\ell}(\delta_{-}(h_{i},t))
\end{align*}
for all $[t_{0},\infty)_{\mathbb{T}}$. Therefore, it follows from Theorem
\ref{thm1} that there exists a constant $M_{0}>1$ such that%
\[
\left\vert x(t)\right\vert \leq M_{0}e_{\lambda}(t,t_{0})\text{ for }%
t\in\lbrack t_{0},\infty)_{\mathbb{T}},
\]
where $\lambda:[t_{0},\infty)_{\mathbb{T}}\mathbb{\rightarrow(-\infty},0)$ is
a positively regressive function defined by (\ref{lambda}). The proof is complete.
\end{proof}

\begin{corollary}
Let $x$ be a function satisfying the inequality%
\begin{equation}
x^{\Delta}(t)\leq-p(t)x(t)+q(t)\max_{i=0,1,...,r}\left\{  x^{\ell}(\delta
_{-}(h_{i},t))\right\}  ,\ \ t\in\lbrack t_{0},\infty)_{\mathbb{T}}\text{,}
\label{3.3}%
\end{equation}
where $\ell\in(0,1]$ is a constant. Suppose that $p$ and $q$ are continuous
and bounded functions satisfying $1-\widetilde{\mu}(t)p(t)>0$ and
$p(t)>q(t)>0$ for all $t\in\lbrack t_{0},\infty)_{\mathbb{T}}$. Then, there
exists a constant $M_{0}>1$ such that every solution $x$ to Eq. (\ref{3.1})
satisfies
\[
\left\vert x(t)\right\vert \leq M_{0}e_{\lambda}(t,t_{0})\text{,}%
\]
where $\lambda$ is a positively regressive function chosen as in (\ref{lambda}).
\end{corollary}

Similar to that of Theorem \ref{thm3.1} one may give a proof of the following
result by using Theorem \ref{thm2} instead of Theorem \ref{thm1}.

\begin{theorem}
\label{thm3.2}Let $p$ and $\beta_{i}$, $i=0,1,...,r$, are continuous functions
satisfying $1-\widetilde{\mu}(t)p(t)>0$, $\beta_{i}(t)>0$, $i=0,1,...,r$, for
all $t\in\lbrack t_{0},\infty)_{\mathbb{T}}$. Assume that $\alpha_{i}%
\in(0,\infty)$, $h_{i}\in\lbrack t_{0},\infty)_{\mathbb{T}}$, $i=0,1,...,r$,
are the scalars such that $\sum_{i=0}^{r}\alpha_{i}=1$, $h_{0}=t_{0}$,
$\delta_{-}(h_{i},t)$, $i=1,...,r$, are the delay functions on $\mathbb{T}$.
If%
\[
p(t)-%
{\displaystyle\prod\limits_{i=0}^{r}}
\beta_{i}(t)>0
\]
for all $t\in\lbrack t_{0},\infty)_{\mathbb{T}}$ and
\[
\left\vert F(t,x(t),x(\delta_{-}(h_{1},t)),...,x(\delta_{-}(h_{r}%
,t)))\right\vert \leq%
{\displaystyle\prod\limits_{i=0}^{r}}
\beta_{i}\left\vert x(\delta_{-}(h_{i},t))\right\vert ^{\alpha_{i}}%
\]
for all $(t,x(t),x(\delta_{-}(h_{1},t)),...,x(\delta_{-}(h_{r},t)))\in\lbrack
t_{0},\infty)_{\mathbb{T}}\times\mathbb{R}^{r+1}$. Then there exists a
constant $N_{0}>1$ such that%
\[
x(t)\leq N_{0}e_{\gamma}(t,t_{0})\text{ for }t\in\lbrack t_{0},\infty
)_{\mathbb{T}},
\]
where $\gamma:[t_{0},\infty)_{\mathbb{T}}\mathbb{\rightarrow(-\infty},0)$ is a
positively regressive function given by (\ref{gamma}).
\end{theorem}

\begin{remark}
\label{rem imp 3}In the case when $\mathbb{T=Z}$, $p(t)=p$ and $q(t)=q$,
Theorem \ref{thm3.1} and Theorem \ref{thm3.2} gives \cite[Theorem 3.1]{udpin}
and \cite[Theorem 3.2]{udpin}, respectively.
\end{remark}






\begin{thebibliography}{00}                                                                                               %
\bibitem {agarwal 1}R. P. Agarwal, \emph{Difference Equations and
Inequalities: Theory, Methods and applications}, Marcel Dekker Inc., New York, 1992.

\bibitem {halanay}A. Halanay, Differential Equations: Stability, Oscillations,
Time lags, Academic Press, New York, NY USA, 1966.

\bibitem {gopalsamy}S. Mohamad and K. Gopalsamy, Continuous and discrete
Halanay-type inequalities, \emph{Bull. Aust. Math. Soc.} 61 (2000), 371--385.

\bibitem {Ivanov}A. Ivanov, E. Liz, and S. Trofimchuk, Halanay inequality,
Yorke 3/2 stability criterion, and differential equations with maxima,
\emph{Tokohu Math.}, 54 (2002), 277--295.

\bibitem {wang}W. Wang, A generalized Halanay inequality for stability of
nonlinear neutral functional differential equations, \emph{Journal of
Inequalities and Applications,} 2010, Article ID 475019.

\bibitem {agarwal}R. P. Agarwal, Y. H. Kim, and S. K. Sen, New discrete
Halanay inequalities: Stability of difference equations,\emph{ Communications
in Applied Analysis,} 12 (2008), 83-90.

\bibitem {agarwal 2}R. P. Agarwal, Y. H. Kim, and S. K. Sen, Advanced discrete
Halanay type inequalities: Stability of difference equations, \emph{Journal of
Inequalities and Applications,} 2009, Article ID 535849.

\bibitem {udpin}S. Udpin and P. Niamsup, New discrete type inequalities and
global stability of nonlinear difference equations, \emph{Appl. Math, Lett.}
22 (2009), 856--859.

\bibitem {liz}E. Liz and J. B. Ferreiro, A note on the global stability of
generalized difference equations, \emph{Appl. Math. Lett.} 15 (2002), 655--659.

\bibitem {baker}C. T. H. Baker, Development and application of Halanay-type
theory: Evolutionary differential and difference equations with time
lag,\emph{ J. Comp. Appl. Math}., 234 (2010), 2663--2682.

\bibitem {baker tang}C. T. H. Baker and A. Tang, 'Generalized Halanay
inequalities for Volterra functional differential equations and discretised
versions', Numerical Analysis Report 229, (Manchester Center for Computational
Mathematics, University of Manchester, England, 1996).

\bibitem {gopalsamy}K. Gopalsamy, \emph{Stability and oscillations in delay
differential equations of population dynamics, }Kluwer Academic Publishers,
Dordrecht, The Netherlands, 1992.

\bibitem {Bi&Bohner}L. Bi, M. Bohner, and M. Fan. Periodic solutions of
functional dynamic equations with infinite delay. \emph{Nonlinear Anal.},
68(2008), 1226--1245.

\bibitem {adivar1}M. Ad\i var and Y. N. Raffoul, Stability and Periodicity in
dynamic delay equations,\emph{ Comput. Math. Appl.} 58 (2009), 264--272.

\bibitem {anderson}D. R. Anderson, R. J. Krueger and A. Peterson, Delay
dynamic equations with stability, \emph{Advances in Difference Equations} vol.
2006 (2006) 19 p. doi:10.1155/ADE/2006/94051 Article ID 94051.

\bibitem {hoffacker}J. Hoffacker and C. Tisdell, Stability and instability for
dynamic equations on time scales, \emph{Comput. Math. Appl.} 49 (2005), pp. 1327--1334.

\bibitem {adrafdelay}M. Ad\i var and Y.N. Raffoul, Shift operators and
stability in delayed dynamic equations, \emph{Rendiconti del Seminario
Matematico Universit\`{a} e Politecnico di Torino}, 68 (2010), 369--397.

\bibitem {kauffmann1}E. Kaufmann and Y.N. Raffoul, Periodicity and stability
in neutral nonlinear dynamic equations with functional delay on a time scale,
\emph{Electron. J. Differential Equations,} 27 (2007), pp. 1--12.

\bibitem {kauffmann2}E. Kaufmann and Y.N. Raffoul, Periodic solutions for a
neutral nonlinear dynamical equations on time scale, \emph{J. Math. Anal.
Appl. }319 (1) (2006), pp. 315--325.

\bibitem {raffoul} E. Ak\i n, Y. N. Raffoul, and C. Tisdell, Exponential stability
in functional dynamic equations on time scales, \emph{Communications in Mathematical Analysis},
9 (2010), 93--108.

\bibitem {raffoul2} E. Ak\i n, Y. N. Raffoul, Boundedness in functional dynamic equations on time scales,
\emph{Advances in Difference Equations}, Vol. 2006, Article ID 79689, Pages 1--18
DOI 10.1155/ADE/2006/79689.

\bibitem {yang cao}Y. Yang and J. Cao, Stability and periodicity in delayed
cellular neural networks with impulsive effects, \emph{Nonlinear Analysis: Real
World Applications}, 8 (2007), 362--374.

\bibitem {hilger}S. Hilger, Analysis on measure chains---a unified approach to
continuous and discrete calculus, \emph{Results Math.} 18 (1990) 18--56.

\bibitem {book}M. Bohner and Allan C. Peterson, \emph{Dynamic equations on
time scales, An introduction with applications.} Birkh\"{a}user Boston Inc.,
Boston, MA, 2001.

\bibitem {book2}M. Bohner and Allan C. Peterson, \emph{Advances in Dynamic
equations on time scales,} Birkh\"{a}user Boston Inc., Boston, MA, 2003.

\bibitem {akin}E. Ak\i n, L. Erbe, B. Kaymak\c{c}alan, and A. Peterson,
Oscillation results for a dynamic equation on a time scale, On the occasion of
the 60th birthday of Calvin Ahlbrandt. \emph{J. Differ. Equations Appl. }(6),
793--810, 2001.

\bibitem {oscillation}M. Bohner, Some oscillation criteria for first order
delay dynamic equations, \emph{Far East J. Appl. Math.}, 18 (2005), 289--304.

\bibitem {akin2}E. Ak\i n Bohner, M. Bohner, and Faysal Ak\i n, Pachpatte
inequalites on time scales, \emph{Journal of Inequalities in Pure and Applied
Mathematics}, 6 (1), Article 6, 2005.

\bibitem {bams}M. Ad\i var and Y. N. Raffoul, Existence of resolvent for
Volterra integral equations on time scales, \emph{Bull. of Aust. Math. Soc.},
82 (2010), 139--155.

\bibitem {adivar}M. Ad\i var, Function bounds for solutions of Volterra
integro dynamic equations on time scales, \emph{E. J. Qualitative Theory of
Diff. Equ.}, 7 (2010), 1--22.




\end{thebibliography}



\end{document}